\documentclass{amsart}

\usepackage{amssymb,xspace}
\usepackage{tikz}
\usepackage{amsthm}
\usepackage[all]{xy}
\usepackage{ifpdf}
\ifpdf\usepackage{hyperref}\else\usepackage[hypertex]{hyperref}\fi

\newtheorem{thm}{Theorem}[section]
\newtheorem{lem}[thm]{Lemma}

\newcounter{x}
\newcounter{y}
\newcounter{z}

\newcommand\xaxis{210}
\newcommand\yaxis{-30}
\newcommand\zaxis{90}

\newcommand\topside[3]{
  \fill[fill=white, draw=black,shift={(\xaxis:#1)},shift={(\yaxis:#2)},
  shift={(\zaxis:#3)}] (0,0) -- (30:1) -- (0,1) --(150:1)--(0,0);
}

\newcommand\leftside[3]{
  \fill[fill=white, draw=black,shift={(\xaxis:#1)},shift={(\yaxis:#2)},
  shift={(\zaxis:#3)}] (0,0) -- (0,-1) -- (210:1) --(150:1)--(0,0);
}

\newcommand\rightside[3]{
  \fill[fill=white, draw=black,shift={(\xaxis:#1)},shift={(\yaxis:#2)},
  shift={(\zaxis:#3)}] (0,0) -- (30:1) -- (-30:1) --(0,-1)--(0,0);
}

\newcommand\cube[3]{
  \topside{#1}{#2}{#3} \leftside{#1}{#2}{#3} \rightside{#1}{#2}{#3}
}

\newcommand\planepartition[1]{
 \setcounter{x}{-1}
  \foreach \a in {#1} {
    \addtocounter{x}{1}
    \setcounter{y}{-1}
    \foreach \b in \a {
      \addtocounter{y}{1}
      \setcounter{z}{-1}
      \foreach \c in {1,...,\b} {
        \addtocounter{z}{1}
        \cube{\value{x}}{\value{y}}{\value{z}}
      }
    }
  }
}

\theoremstyle{plain}

\newtheorem{ex}[thm]{Example}

\numberwithin{equation}{section}
    \newtheoremstyle{TheoremNum}
        {\topsep}{\topsep}              
        {\itshape}                      
        {}                              
        {\bfseries}                     
        {.}                             
        { }                             
        {\thmname{#1}\thmnote{ \bfseries #3}}
    \theoremstyle{TheoremNum}

\newcommand{\C}{\mathbb{C}}

\DeclareMathOperator{\im}{im}

\DeclareMathOperator{\Hilb}{Hilb}

\newif\ifrem\remtrue

\def\DMO{\DeclareMathOperator}

\DMO{\Exp}{Exp}
\DMO{\Aut}{Aut}
\DMO{\Pow}{Pow}
\def\Re{\operatorname{Re}}

\def\im{\mathrm{im}}

\begin{document}

\title{A Gaussian distribution for refined DT invariants and 3D partitions.}
\author[Andrew Morrison]{Andrew Morrison}
\begin{abstract} We show that the refined Donaldson--Thomas invariants of $\C^3$, suitably normalized, have a Gaussian distribution as limit law. Combinatorially these numbers are given by weighted counts of 3D partitions. Our technique is to use the Hardy--Littlewood circle method to analyze the bivariate asymptotics of a $q$-deformation of MacMahon's function. The proof is based on that of E.M. Wright who explored the single variable case.
\end{abstract}

\maketitle

\section{Introduction.}
In \cite{hm} physicists suggested that the (IIA) string theory associated to a Calabi--Yau threefold $X$ should produce an algebra of BPS states:
\[ \mathcal{H}_{BPS} = \bigoplus_{\gamma\in \Gamma} \mathcal{H}_\gamma  \]
where here the (charge) lattice $\Gamma$ is identified with the even cohomology of $X$, $\Gamma = \bigoplus_{i=0}^3 H^\textrm{2i}(X,\mathbb{Z})$. Moreover each individual vector space $\mathcal{H}_\gamma$ should have an additional $\mathbb{Z}$-grading coming from a symmetry in the little group $\textrm{Spin}(3)$ \cite{f}.
\\ \\
Mathematically, we consider the cohomological Hall algebra \cite{ks} as giving the algebra of BPS states on $X= \C^3$. In this case $\Gamma = H^6(X,\mathbb{Z}) = \mathbb{Z}$ and the $\gamma=n$th piece is given by the critical cohomology of $\Hilb^n(\C^3)$\footnote{In addition we add a single D$6$-brane filing the $\C^3$ or mathematically a framing.}. Moreover each of these vector spaces has a cohomological $\mathbb{Z}$-grading. The Betti numbers of these graded pieces are know as refined DT invariants. These numbers are dependent on the singularities of $\Hilb^n(\C^3)$. For example they do not satisfy a Poincare duality. 
\\ \\
However, on a recent visit to EPFL, T. Hausel shared with me the output of a computer experiment. He conjectured that the refined DT invariants, suitably normalized, would have a Gaussian distribution as limit law, i.e. for large $n$ plotting the Betti numbers against cohomological degree should give the bell curve of a Gaussian distribution (cf. \cite{r2}). The goal of this paper is to prove that conjecture.
\\ \\
In fact this proposal is entirely combinatorial. The Hilbert-Poincare series for the cohomological Hall algebra, computed in \cite{bbs}, equals $M_3(t,q^{1/2})$ where $t$ gives the $\Gamma$-grading, $q$ gives the cohomological grading, and
\[ M_\delta(t,q) = \prod_{m\geq 1} \prod_{k=0}^{m-1} \frac{1}{1-q^{\delta + 2k + 1 - m}t^m} .\]
Expanding this series gives an explicit formula \cite{or} for the $t^n$ coefficent 
\[ \sum_{\pi \vdash n} q^{\delta w_0(\pi) + w_+(\pi) - w_-(\pi)} \]
where the sum is over all plane partitions of $n$ which we now explain.

A plane partition is given by a two dimensional array of positive integers in the first quadrant of $\mathbb{Z}^2$ that are weakly decreasing in both the $x,y$ directions. In this way plane partitions are a generalization of ordinary (line) partitions \cite{a}. The analogue of the Young diagram in this situation is a stack of three dimensional boxes in $\mathbb{Z}^3_{\geq 0}$. Such a collection gives a plane partition if and only if the stack is stable under the pull of gravity along the $(1,1,1)$ axis. For example let $\pi$ be the plane partition given alternatively as

\begin{figure}[ht]
	\begin{minipage}[b]{0.45\linewidth}
	\centering
	$\begin{array}{ccccccc}
	\vdots & \vdots & \vdots & \vdots &  &  & \\
0 & 0 & 0 & 0 &     &  & \\
2 & 2 & 1 & 0 & 0 & \cdots & \\
3 & 3 & 2 & 0 & 0 & \cdots & \\
4 & 3 & 2 & 1 & 0 & \cdots &  \\
5 & 3 & 2 & 1 & 1 & 0 & \cdots 
  \end{array}$
	\caption{integers}
	\label{fig:figure1}
	\end{minipage}
	\hspace{0.5cm}
	\begin{minipage}[b]{0.45\linewidth}
	\centering
	\begin{tikzpicture}[scale=0.3][center]\planepartition{{5,3,2,1,1},{4,3,2,1},{3,3,2},{2,2,1}}\end{tikzpicture}
	\caption{boxes}
	\label{fig:figure2}
	\end{minipage}
	\end{figure}

Here the total sum of the integers/boxes is $|\pi| = 35 $ and we say that the partition has size $35$. The statistics appearing in the above formula for refined DT invariants are,

\[w_+(\pi) = \sum_{i<j} \pi_{i,j}, \hspace{1.0cm}
w_-(\pi) = \sum_{i>j} \pi_{i,j}, \hspace{1.0cm}
 w_0(\pi) = \sum_i \pi_{i,i}.\]

Considering the set $\mathcal{P}_n$ of plane partitions of size $n$ as a sample space with uniform measure we define three random variables
\[ X_{n}^+,X_{n}^-,X_{n}^0 : \mathcal{P}_n \to \mathbb{R} \]
given by $w_+/n^{2/3},w_-/n^{2/3},$ and $w_0/n^{2/3}$. Our main result is the following:
\begin{thm}\label{main'}The distribution of the random vairable
\[ \delta \cdot X_{n}^0 + X_{n}^+ - X_{n}^- \]
for large $n$ has the Gaussian distribution as limit law with
\[ \mu = \delta \zeta(2) /(2\zeta(3))^{2/3} \textrm{ and } \sigma^2 = 1/(2\zeta(3))^{1/3} .\] 
\end{thm}

\section{Setup.}
First we split the problem into two parts, one of which has already been solved. Straight away we see that the covariance of $X_{n}^0$ and $X_{n}^+ - X_{n}^-$ is zero due to symmetry
\begin{eqnarray*} 
\mathbb{E}((X_{n}^0 - \mu_{X_{n}^0})(X_{n}^+ - X_{n}^- - \mu_{X_{n}^+ - X_{n}^-} )) & = & \mathbb{E}((X_{n}^0- \mu_{X_{n}^0})(X_{n}^+ - X_{n}^- )) \\ 
& = & \mathbb{E}(X_{n}^0(X_{n}^+ - X_{n}^-))  - \mathbb{E}(\mu_{X_{n}^0}(X_{n}^+ - X_{n}^-)) \\ 
& = & \mathbb{E}(X_{n}^0 X_{n}^+ - X_{n}^0 X_{n}^-) -\mu_{X_{n}^0}(\mathbb{E}(X_{n}^+ - X_{n}^-))\\ 
& = & 0 - 0.
\end{eqnarray*}
The following is a result of E.P. Kamenov and L.R. Mutafchiev:
\begin{thm}[\cite{km}]Let $ a= \zeta(2)/ (2\zeta(3))^{2/3} $ and $ b = \sqrt{1/3} / (2\zeta(3))^{1/3}$. Then as $n\to \infty$ we have
\[ X^0_n  \sim \mathcal{N}(a,bn^{-\frac{1}{3}}\ln^{\frac{1}{2}}n) .\]
\end{thm}
Now the sum of two Gaussian random variables is again Gaussian with mean the sum of the means and variance the sum of the variances plus a covariance term. Since the above covariance was zero, Theorem \ref{main'} will now follow from the result just mentioned together with:
\begin{thm}\label{main} Let $c = 1/(2\zeta(3))^{1/3}$. Then as $n\to \infty$ we have
\[ X_{n}^+ - X_{n}^- \sim \mathcal{N}(0,c). \]
\end{thm}
To prove this result we use the method of moments. That is we show that the limiting distribution has the same moments as a Gaussian random variable with variance $1/(2\zeta(3))^{1/3}$. Specifically we will show that in the limit
\[ \mathbb{E}\left(\left( X_{n}^+ - X_{n}^- \right)^{k}\right) = \left\{ \begin{array}{cl} 0 & \textrm{ if $k$ is odd, } \\  (k-1)!! (2\zeta(3))^{-k/6} & \textrm{ if $k$ is even. } \end{array} \right. \]
Consider the generating series 
\[ M_0(t,q) = \prod_{m\geq 0} \prod_{k=0}^{m-1} \frac{1}{1-q^{2k+1-m}t^m} = \sum_{\pi} q^{w_+(\pi)-w_-(\pi)}t^{|\pi|} \]
and let $p_n(q)$ be the coefficient of $t^n$ then we have 
\[ \mathbb{E}\left(\left( X_{n}^+ - X_{n}^- \right)^{k}\right) = n^{-2k/3} \frac{\left. \partial^k p_n(q)\right|_{q=1} }{p_n(q)|_{q=1}} \]
where $\partial = q\frac{d}{dq}$. Notice by symmetry this already implies that all the odd moments vanish. The method of proof given in the next section follows the proof of E.Wright \cite{w} who provided an asymptotic formula\footnote{Typo warning: Wrights formula on page 179 is missing factor of $\sqrt{3}$ found at the end of his proof on page 189.} for the number of plane partitions of $n$
\[ p_n(q)|_{q=1} \sim \frac{\zeta(3)^{7/36}}{2^{11/36}\sqrt{3\pi}n^{25/36}}e^{ 3\left(\frac{\zeta(3)}{4}\right)^{1/3}(n^{4/3})^{1/2} + \zeta'(-1)}. \]
Wright's proof in turn generalized the pioneering work of Hardy and Ramanujan \cite{hr} who first applied the Hardy-Littlewood circle method to get an asymptotic formula for the number of ordinary partitions. Using this method in the next section we will show that 
\[ \partial^k p_n(q) |_{q=1} \sim n^{2k/3} \cdot (k-1)!! (2\zeta(3))^{-k/6} \cdot \frac{\zeta(3)^{7/36}}{2^{11/36}\sqrt{3\pi}n^{25/36}}e^{ 3\left(\frac{\zeta(3)}{4}\right)^{1/3}(n^{4/3})^{1/2} + \zeta'(-1)} \]
when $k$ is even. This gives the correct moments and shows that $X_{n}^+ - X_{n}^-$ has a Gaussian limit law as promised. 

\section{Proof.}
As explained in the previous section we are going to use the Hardy--Littlewood circle method to estimate the coefficients in the generating series
\[ M_k(t) := \partial^k M_0(t,q) |_{q=1}. \]
Given any function $A(t) = \sum_{n\geq 0} a_nt^n$ analytic on the interior of the unit disk we can compute the $n$th coefficient in its MacLaurin series using the Cauchy formula
\[ a_n = \frac{1}{2\pi i } \int_{C_N} t^{-n-1} A(t)dt   \]
where $C_N$ is the circle or radius $e^{-1/N}$. The idea of the circle method is that by understanding the singularities of $A$ on the unit circle one can approximate this integral by an integral over a small subarc of the circle when $N,n \gg 1$. 
\\ \\
For example when $A(t) = M(t) = M(t,q)|_{q=1}$ is MacMahon's function then letting $t=e^z$ Wright defines the \textit{major arc} $C'_N$ to be the points such that $\im(z)<1/N$ and the \textit{minor arc} $C''_N$ to be the remaining points on the circle $C_N$. In some sense MacMahon's function is most singular at $z=0$ and so the integral over $C_N$ is well approximated by that over the small arc $C'_N$. The following two Lemmas\footnote{Wright's first Lemma is more refined than this. We mearly extract his leading order approximation suitable for our purposes.} make this precise and will be very useful to us later:
\begin{lem}[E.M.Wright \cite{w} Lemma I]\label{majorM} There exists constants $N_0,K$ so that for all $N>N_0$ and $t = e^z \in C'_N$ along the major arc we have \[ \left| M(t) - e^{-\zeta'(1)}z^{1/12} e^{\zeta(3)/z^2}\right| < K N^{-1/12-2} e^{\zeta(3)N^2}. \] 
\end{lem}
\begin{lem}[E.M.Wright \cite{w} Lemma II]\label{minorM} Given any $\epsilon > 0$ there exists an $N_0$ such that for all $N>N_0$ and $t\in C''_N$ along the minor arc we have
\[ |M(t)| < e^{(\zeta(3) -1/2 +\epsilon)N^2} . \]
\end{lem}
Using Lemma \ref{minorM} one shows that the integral along the minor arc is relatively small. Then using Lemma \ref{majorM} the integral along the major arc can be approximated using the curve of steepest decent. This gives Wright's asymptotic formula mentioned earlier \cite{w} and illustrates the idea of the circle method.
\\ \\ 
Recall, we are specifically interested in computing the coefficients of the series $M_k(t) = \partial^k M_0(t,q) |_{q=1}$ defined at the start of the section as a means to compute the moments of the random variables $X^+_n-X^-_n$. Let us write this series as $M_k(t) = F_k(t)\cdot M(t)$ then by Wright's two lemmas above we have a good understanding of the singularities of the factor $M(t)$ it remains to analyze $F_k(t)$. 
\begin{ex}\label{exF2}Computing $F_2(t)$. Let us differentiate $M_0(t,q)$ twice using $\partial = q\frac{d}{dq}$ this gives 
\[ \left( \sum_{m\geq 1} \sum_{k=0}^{m-1} \frac{(1-m+2k)^2q^{4k-2m+2}t^{2m}}{(1-q^{2k+1-m}t^m)^2} + \frac{(1-m+2k)^2q^{2k-m+1}t^{m}}{(1-q^{2k+1-m}t^m)}  \right)\cdot M_0(t,q)\]
\[ + \left( \sum_{m\geq 1} \sum_{k=0}^{m-1} \frac{ (1-m+2k)q^{2k-m+1}t^m }{1- q^{2k-m+1}t^m} \right)^2\cdot M_0(t,q). \]
Then setting $q=1$ we get
\[\left(\sum_{m\geq 1} \sum_{k=0}^{m-1} \frac{(1+2k-m)^2t^m}{(1-t^m)^2} \right)\cdot M(t) = \frac{1}{3}\cdot\left(\sum_{m\geq 1} \frac{m(m^2-1)t^m}{(1-t^m)^2}\right) \cdot M(t) \]
deducing that
\[ F_2(t) = \frac{1}{3}\sum_{m\geq 1} \frac{m(m^2-1)t^m}{(1-t^m)^2}. \]
\end{ex}
In the next two Lemmas we analyze the behavior of $F_k(t)$ along the major and minor arcs.
\begin{lem}\label{majorF}Given $k$ even, there exist constants $N_0,K$ such that for all $N>N_0$ and $t=e^z\in C'_N$ along the major arc we have
\[ \left|F_k(t) - (k-1)!!(2\zeta(3))^{k/2}z^{-2k} \right| < K N^{2k-2}. \] 
\end{lem}
\begin{proof} First let us consider the case $k=2$ used to compute the variance. By what we saw in Example \ref{exF2} above
\[ F_2(t) = \frac{1}{3} \sum_{m\geq 1} \frac{m(m^2-1)e^{mz}}{(1-e^{mz})^2} .\]
By using the Mellin transform $e^{-\tau} = \frac{1}{2\pi i }\int_{\sigma - i\infty}^{\sigma + i\infty}\Gamma(s)\tau^{-s}ds$ we can express this function as an integral,
\begin{eqnarray*}
F_2(t) & = & \frac{1}{3}\sum_{m\geq 1} \frac{m(m^2-1)e^{mz}}{(1-e^{mz})^2}
=  \frac{1}{3}\sum_{m\geq 1} \sum_{i\geq 1} m(m^2-1) i e^{imz} \\
& = &  \frac{1}{6i\pi}\int_{\sigma - i \infty}^{\sigma + i \infty} \Gamma(s)\sum_{m\geq 1} \sum_{i\geq 1} \frac{m (m^2-1) i }{ (imz)^{s}} ds \\
& = &  \frac{1}{6i\pi}\int_{\sigma - i \infty}^{\sigma + i \infty} \Gamma(s) z^{-s} \zeta(s-1)(\zeta(s-3) -\zeta(s-1)) ds.
\end{eqnarray*}
This integrand has a double pole at $s=2$ and a simple pole at $s=4$ coming from the Riemann zeta function while the gamma function contributes simple poles at $s=0,-1,-2,\ldots$ . Doing the residue calculus we see that 
\[ F_2(t) = \frac{\Gamma(4) \zeta(3)}{3}  z^{-4} -\frac{2\gamma \Gamma(2) \zeta(-1)}{3}z^{-2} + O(1) \]
where $\gamma $ is Euler's constant coming from the Laurent expansion of the zeta function about $s=1$. 

Computing the higher order derivatives is essentially an application of Wick's theorem. Indeed if $k=2r$ then we have 
\[ F_{k=2r}(t) = (k-1)!! (F_2(t))^r + G_k(t)  \]
where the combinatorial coefficient $(k-1)!!$ comes from all possible ways of pairing the $2r$ differentials $\partial^{2r}$. Using the product rule for differentiation we see the first term appearing. The remaining terms in $G_k(t)$ come from the other terms generated in the product rule. Computing their Mellin transform shows the poles they contribute are of order at least two less. 
\end{proof}

\begin{lem}\label{minorF} Given k even, there exists positive constants $N_0,C_k,A_k$ such that for all $N>N_0$ and $t=e^z\in C''_N$ along the minor arc we have
\[ \left| F_k(t) \right| <  C_kN^{A_k}.\]
\end{lem}
\begin{proof} Looking at the definition of $F_k(t)$ we see that ultimately it can be written as a finite sums and products of series like 
\[ \sum_{n\geq 1} \frac{ n^a t^{nc} }{(1-t^n)^b}  \]
where $a,b,c \in \mathbb{N}$. Each of these sums can be bounded like 
\[ \left|\sum_{n\geq 1} \frac{ n^a t^{nc} }{(1-t^n)^b} \right |\leq\sum_{n\geq 1} \frac{ n^a |t|^{n} }{(1-|t|^n)^b}\]
using the Mellin transform as in the previous lemma gives bounds on this sum like $C_{a,b}N^{A_{a,b}}$ where $|t| = e^{-1/N}$ and $A_{a,b}, C_{a,b}$ are constants depending only on $a,b$. In total this gives the polynomial bounds claimed.
\end{proof}
Now we have got bounds on the series $M_k(t)$ along the major and minor arcs we are ready to estimate the Cauchy integral. Let us define the two quantities we are interested in computing
\begin{eqnarray*} I'_{k,n} & = & \frac{1}{2\pi i }\int_{C'_N} t^{-n-1}M_k(t) dt, \\ 
I''_{k,n} & = & \frac{1}{2\pi i }\int_{C''_N} t^{-n-1}M_k(t) dt .
\end{eqnarray*}
From now onwards we choose $N = (n/2\zeta(3))^{1/3}$ so that $n=2\zeta(3)N^3$. Now we can get a bound on the integral along the minor arc that is sufficient to show this integral is negligible:
\begin{lem} Given $k$ even then for all $\epsilon >0$ there exists positive constants $N_0,K,A$ such that for all $N>N_0$ we have 
\[ |I''_{k,n}| < K\cdot N^A \cdot e^{(-\frac{1}{2} + \epsilon)N^2} \cdot e^{3\zeta(3)N^2}. \]
\end{lem}
\begin{proof} Using Lemmas \ref{minorM} and \ref{minorF} we get
\begin{eqnarray*}
|I''_{k,n}| & = &\left| \frac{1}{2\pi i }\int_{C''_N}F_k(t)\cdot M(t) \cdot t^{-n-1} dt \right| \\
 & \leq & \frac{2\pi}{|2\pi i |} \cdot \sup_{C''_N}( |F_k(t)|) \cdot \sup_{C''_N}( |M(t)|) \cdot  \sup_{C''_N}( |t^{-n-1}|)\\
  & \leq & K \cdot N^A \cdot e^{(\zeta(3) -1/2 +\epsilon)N^2} \cdot e^{(-n-1) N^{-1} }
\end{eqnarray*}
substituting $n=2\zeta(3)N^3$ gives the result.
\end{proof}
From now all that remains is to estimate the integral $I'_{n,k}$. Combining Lemma \ref{majorM} and Lemma \ref{majorF} we have
\begin{eqnarray*}
I'_{n,k} & = & \frac{1}{2\pi i} \int_{  C'_N } F_k(t)M(t)t^{-n-1} dt \\
             & = & \frac{e^{\zeta'(-1)}}{2\pi i} \frac{(k-1)!!}{(2\zeta (3))^{k/6}} \int_{(1-i)/N}^{(1+i)/N} z^{- 2k+1/12 }e^{\frac{\zeta(3)}{z^2} + 2\zeta(3) N^3 z }dz \\ 
             & & \hspace{4.5cm}+ \hspace{0.5cm} \mathcal{O}( N^{-1 +2k -1/12 - 2}e^{3\zeta(3)N^2})\\
             & = & \frac{(k-1)!! N^{2k}}{(2\zeta (3))^{k/6}} \cdot \frac{e^{\zeta'(-1)}N^{-1/12-1}}{2\pi i} \int_{1-i}^{1+i} v^{- 2k+1/12 }e^{ \zeta(3)N^2(2v + v^{-2})}dv \\ 
             & & \hspace{4.5cm}+ \hspace{0.5cm} \mathcal{O}( N^{-1 +2k -1/12 - 2 }e^{3\zeta(3)N^2})              
\end{eqnarray*}
where we set $v=Nz$. 

Notice the prefactor in the above expression is essentially the term we are looking for. From now on we work with the integral
\[ P'_{n,k} := \frac{1}{2\pi i} \int_{1-i}^{1+i} v^{- 2k+1/12 }e^{ \zeta(3)N^2(2v + v^{-2})}dv  \]
basically it will be enough to show that, for large $N$, this is independent of $k$.  Then in the limit we have $I'_{n,k} =  \frac{(k-1)!! N^{2k}}{(2\zeta (3))^{k/6}} \cdot I'_{n,0}$. We are able to achieve this by localizing the integral $P'_{n,k}$ to an even smaller arc using the method of steepest descents. 

In the exponent of the integrand we have the function $2v+v^{-2}$. Roughly speaking the integrand will be largest when this exponent is real. The curve of steepest descent is defined to be the real curve given by $\im( 2v+v^{-2} ) = 0$, specifically taking $v = X + iY$ we have
\[ (X^2+Y^2)^2 = X. \]
This is the closed curve $\mathcal{C}$ meeting the lines $X=0$ at $(0,0)$, $X=1$ at $(1,0)$, $X=-Y$ at $D= (2^{-\frac{2}{3}},-2^{-\frac{2}{3}})$, and $X=Y$ at $E = (2^{-\frac{2}{3}},2^{-\frac{2}{3}})$ as seen in Figure \ref{steep}. As in \cite{w} we consider an alternative integral along this curve $\mathcal{C}$ rather than along the straight line $FG$. 

Making a branch cut from $0$ to $-\infty$ along the real axis we consider the value of $v^{1/12}$ which is real and positive at $v=1$ and take the contour for $\mathcal{C}$ parameterized in the anti-clockwise direction. Following Wright we define 
\[ \xi_k(v) = \frac{v^{-2k+1/12}}{2\pi i} e^{\zeta(3) N^2 (2v + v^{-2})} \textrm{ and } J'_{k,n} = \int_{\mathcal{C}}\xi_k(v) dv .\]

 \begin{figure}
 \includegraphics[scale = 0.25]{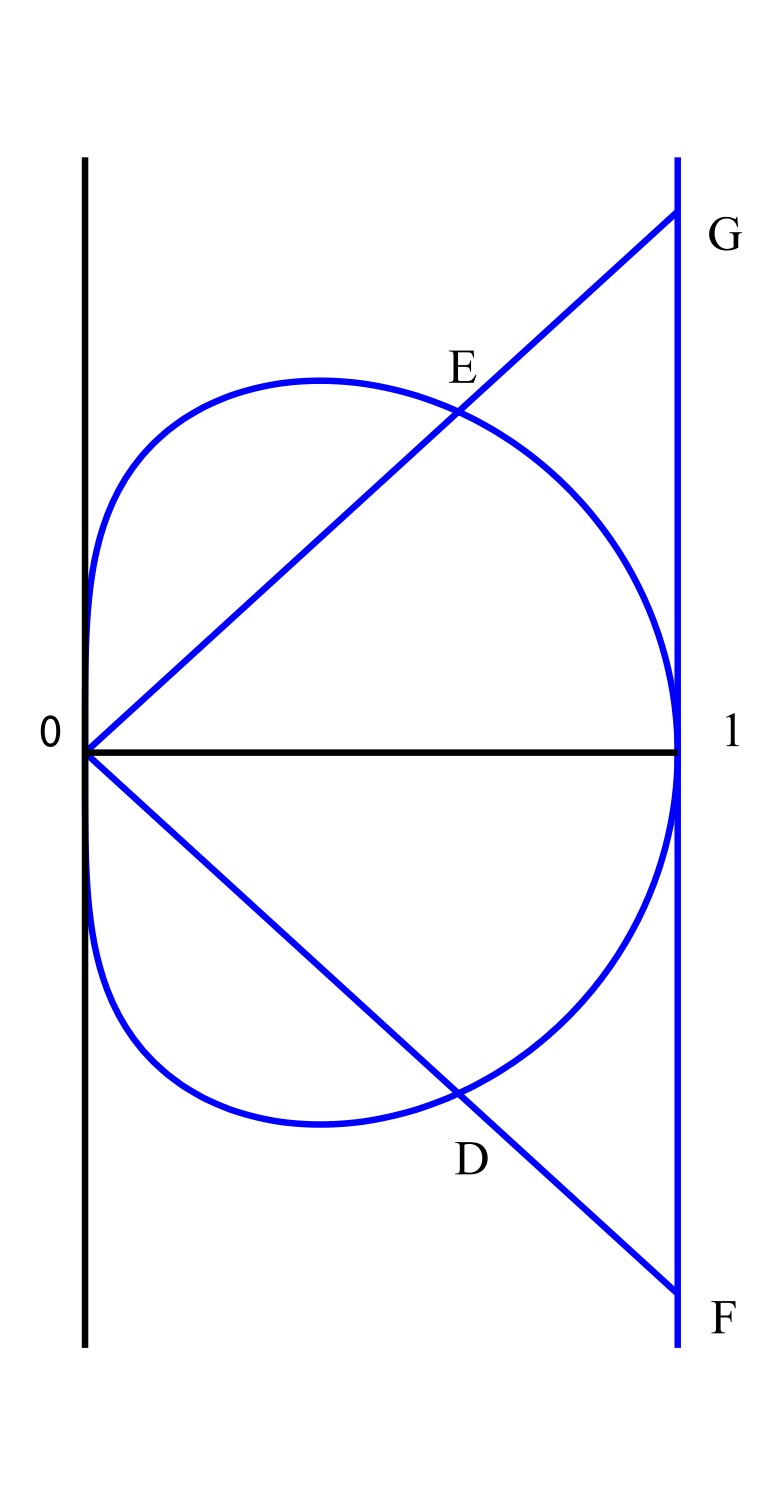}
 	\caption{Curve of steepest descent.}
	\label{steep}
\end{figure}

On the straight lines $EG,DF$ and along the arcs $OE$ and $OD$ we have good bounds on $\xi_k(v)$. Since here $\Re ( v^{-2})= (X^2-Y^2)/(X^2+Y^2)^2 \leq 0$ setting $Y= rX$ gives
\[ \left|v^{-2k +1/12}e^{\zeta(3)N^2( 2v + v^{-2})}\right| \leq {X^{-2k}}(1+r^2)^{-k} e^{\zeta(3)N^2\left(\frac{1-r^2}{(1+r^2)^{2}}X^{-2}\right)}  e^{2\zeta(3) N^2 X } . \]
When $r >  1$ this tends to zero as $X \to 0$, and when $r=1$ along the lines $EG,DF$ there are easy bounds. In summary, we have a bound $|\xi_k(t)| < K e^{2\zeta(3)N^2 X }$ along these contours $OE,OD,EC,$ and $DF$. Using contour integration to compare the original integral $P'_{n,k}$ to the new integral $J'_{n,k}$ along the curve $\mathcal{C}$ we see that
\[ \left| P'_{n,k} - J'_{n,k} \right| < \left|\int_{F}^D  \xi_k(v) dv\right| +\left|\int_{D}^0 \right| +\left|\int_{0}^E \right| +\left|\int_{E}^G \right| < Ke^{2\zeta(3)N^2} \]
this allows us to integrate along the curve of steepest descent instead. To parameterize this curve we choose
\[ t = -i\left( \frac{v-1}{v} \right) (2v + 1)^{\frac{1}{2}} \]
so that $t^2 = 3-2v - v^{-2}$. Now the problem transforms to an integral over the real line
\[ J_{n,k}' = e^{3\zeta(3) N^2} \int_{-\infty}^{\infty} \chi_k(t) e^{-\zeta(3) N^2t^2 } dt,  \textrm{ with } \chi_k(t) = \frac{v^{-2k +1/12}}{2\pi i} \frac{dv}{dt}.\]
The most serious piece of this integral is located about $t=0$ i.e. $v=1$.  To understand the behavior here we take a convergent power series $\chi_k(t) = \sum_{m=0}^\infty a_mt^m $ in a small neighborhood of $t=0$. Next observe that for some constant $K$ we have
\[ \frac{ dv }{dt} = \frac{iv^2(2v+1)^{\frac{1}{2}} }{ 1 + v + v^2 } \Rightarrow \left| \frac{v^{-2k +1/12}}{2\pi i} \frac{dv}{dt}\right| < K|v^{-2k}| \]
on all of the real line. Moreover for some possibly larger $K$ we have 
\[ | v^{-2k} | = | 3-2 v- t^2 |^k \leq Kt^{2k} \]
on the compliment of the above radius of convergence about $t=0$. All in all we get that 
\[ \left| \chi_k(t) - \sum_{m=0}^{2k+1} a_m t^m \right| < Kt^{2k + 2} \]
over the whole real line. Finally our integral is approximated by 
\[  J_{n,k}' = \sum_{m=0}^{2k+1} a_m\int_{-\infty}^{\infty} t^me^{-\zeta(3) N^2 t^2}dt + M_k \]
where
\[ \left| M_k \right| \leq K \int_{-\infty}^{\infty} t^{2k+2}e^{-\zeta(3) N^2 t^2}dt  < \frac{K}{N^{2k+3}} .\]
By symmetry all the above odd integrals are zero. The even ones are given by
\[ \int_{-\infty}^{\infty} t^{2m} e^{\zeta(3) N^2 t^2} dt  = \frac{\Gamma(m+\frac{1}{2})}{(\zeta(3) N^2)^{m+\frac{1}{2}}}. \]
So to get the leading order asymptotics we need only the constant term $a_0 = \frac{1}{2\pi \sqrt{3}}$ in the expansion of $\chi_k(t)$. In particular to leading order there is no dependence on $k$ as we claimed earlier. Substituting $N= (n/2\zeta (3) )^{1/3}$ gives the formula described at the end of Section 2.

\section{Final Remarks.}
Some remarks about the asymptotics of DT invariants coming from this investigation and math/physics literature.  
\\
\paragraph{\textbf{Asymmetry}.} Since the refined DT invariants are given by $M_3(t,q^{1/2})$ their distribution is shifted by the trace of the plane partition coming from $3X^0_n$. This shifting is relevant to the refined topological vertex in physics \cite{ikv}. This was discussed in \cite{mmns} where a minor discrepancy was noticed with calculations in \cite{dg} for the case of the refined DT invariants of the resolved conifold singularity. 
\\
\paragraph{\textbf{Dimension of moduli space}.}In his lecture notes on quiver moduli M. Reineke describes a conjecture on M. Douglas on the asymptotic growth of Euler numbers of spaces of representations of Kronecker quivers \cite{r}. Reineke proposes a generalization of the conjecture would give 
\[ \ln \left( \chi(M_d(Q)) \right) \sim C_Q\sqrt{ \dim(M_d(Q)) } \] 
where $M_d(Q))$ is a suitable smooth model for the quiver moduli, $d$ is a large dimension vector, and $C_Q$ is an interesting constant to be determined. 

For sheaves on a Calabi-Yau threefold the moduli spaces will be singular and not of the expected dimension. However in our case we do know that
\[C^-_{\C^3}\sqrt{ \dim(\Hilb^n(\C^3)}) <  \ln \left(| \#^{\textrm{vir}}(\Hilb^n(\C^3))| \right) <  C^+_{\C^3}\sqrt{ \dim(\Hilb^n(\C^3)})   \]
for large $n$, where $\#^{\textrm{vir}}$ is the virtual Euler number or numerical DT count for this moduli space, and $C^-_{\C^3},C^+_{\C^3}$ are constants. A proof of this follows from Wrights theorem \cite{w} for the virtual Euler number and Briancon and Iarrobino's asymptotics for the dimension of $\Hilb^n(\C^3)$ \cite{bi}. Geometrically this relationship between the virtual Euler number and dimension of the moduli space seems like a strange coincidence specific to $\C^3$?
\\ 
\paragraph{\textbf{Orbifold: $\C \times [\C^2/\mathbb{Z}_2]$}.} In the case of numerical DT invariants, Panario, Richmond, and Young investigated the orbifold $\C \times [\C^2/\mathbb{Z}_2]$, here the charge lattice is two dimensional and one studies the bivariate asymptotics of colored partitions see \cite{pry}. 
\\
\paragraph{\textbf{BPS black holes}.} A major goal of string theory is to unify quantum mechanics with Einstein's general relativity. In the 90's Strominger and Vafa showed that indeed the topological string described the physics of some BPS black holes in a certain limit of the string coupling constant \cite{sv}. Later with Ooguri they conjectured that in a certain limit the entropy of BPS black holes should be determined from the square of the topological string partition function \cite{osv}.

Recently in the case of D6 - D2/D0 states this conjecture has been developed by Denef and Moore \cite{dm}. If we let $X$ be a Calabi-Yau threefold and $I(X,\beta,n)$ be the moduli space of ideal sheaves with Chern character $(1,0,-\beta,n)$, their paper speculates that 
\[ \lim_{\gamma \to \infty} \left( \frac{\ln \left( \ln \left( |\#^\textrm{vir} (I(X,\lambda^2 \beta,\lambda^3 n)) | \right) \right)}{\ln (\lambda)}\right) = 2. \]
By Wright's theorem this is true when $\beta = 0$. In \cite{hkmt} physicists checked this conjecture for the quintic threefold. However as it is a hard problem to compute the higher genus Gromov-Witten invariants in this case these checks are only partial.
\section*{Acknowledgements.} 
Primarily I wish to thank T. Hausel for his hospitality in inviting me to EPFL and sharing the ideas that lead to this paper. Also thanks to J. Bryan, R. Pandharipande, and B. Young for their support and comments, and to M. Marcolli who was my mentor at MSRI where this paper was written. 

This paper was finished during a postdoc at MSRI in the spring $2013$ program Non-Commutative Algebraic Geometry and Representation Theory. Normally I am a postdoc at ETH Z\"{u}rich in the research group of R. Pandharipande sponsored by Swiss grant 200021143274.

Email : andrewmo@math.ethz.ch
\end{document}